\documentclass[10pt]{article}
\usepackage{color}
\usepackage{amssymb}
\usepackage{amsthm,array,amssymb,amscd,amsfonts,latexsym, url}
\usepackage{amsmath}
\usepackage[all]{xy}
\newtheorem{theo}{Theorem}[section]
\newtheorem{prop}[theo]{Proposition}
\newtheorem{claim}[theo]{Claim}
\newtheorem{lemm}[theo]{Lemma}
\newtheorem{coro}[theo]{Corollary}
\newtheorem{question}[theo]{Question}
\newtheorem{rema}[theo]{Remark}
\newtheorem{Defi}[theo]{Definition}

\title{Compact K\"{a}hler  manifolds with no projective specialization}
\author{Claire Voisin\footnote{The author is supported by the ERC Synergy Grant HyperK (Grant agreement No. 854361).}}

\date{}

\newfont{\gothic}{eufb10}
\begin{document}
\maketitle
\setcounter{section}{-1}
\begin{flushright} {\it Pour Fabrizio,  \`{a} l'occasion de son 70\`{e}me anniversaire}
\end{flushright}

\begin{abstract} We  show the existence of a compact K\"{a}hler manifold which does not fit in a proper flat  family over an irreducible base with one  projective (possibly singular)   fiber.  We also give a topological version of this statement. This strengthens our earlier counterexamples to the Kodaira algebraic approximation problem.
 \end{abstract}
\section{Introduction}
In the paper \cite{kodaira}, Kodaira proved that any compact K\"ahler surface admits ``algebraic approximations", that is, it becomes projective by an arbitrarily small deformation. This result, obtained as a consequence of the classification of surfaces, has been reproved in \cite{buchdahl} by an infinitesimal method. A similar statement has been obtained recently by Lin \cite{lin} for threefolds. However, starting from dimension 4, we constructed  in \cite{voisin}  compact K\"{a}hler manifolds not homeomorphic to, and in particular not deformation equivalent  to, a complex projective manifold. In the paper \cite{voisinjdg}, we also exhibited examples of  compact K\"{a}hler manifolds,  no  smooth bimeromorphic models of which is  homeomorphic to a complex  projective manifold.

The manifolds $X$ constructed in \cite{voisin} thus have the property that there is no smooth proper holomorphic map
$\mathcal{X}\rightarrow B$, where $B$ is a connected analytic space, with two points $b,\,b'\in B$  such that
$\mathcal{X}_b\cong X$ and the fiber $\mathcal{X}_{b'}$ is projective. This does not say anything however on the singular specializations $\mathcal{X}_{b'}$ of such proper flat  holomorphic maps   $\mathcal{X}\rightarrow B$ with one  fiber  $\mathcal{X}_b$  isomorphic to $X$.
\begin{Defi}    A compact complex manifold has a projective specialization if there exists a proper flat  holomorphic  map  $\mathcal{X}\rightarrow B$, where $B$ is an irreducible    analytic space, with one fiber $\mathcal{X}_b$ isomorphic to $X$ and one projective  fiber $\mathcal{X}_0$.
\end{Defi}
The  main result of this paper is the  following theorem, which addresses a question  asked to us by A. Kazhymurat.

\begin{theo}\label{theomain}   There exists a compact K\"ahler manifold $X$ with no projective specialization. Such manifolds exist starting from dimension $5$.
\end{theo}
A compact K\"ahler  manifold  $X$ satisfying this property  is constructed in Section \ref{sec2X}. Note that the result  presumably holds   in  dimension 4 as well.

We will also establish    in Section \ref{secfurther}   a variant
 of these results,   using stronger properties  of  the  cohomology algebra of the considered compact K\"ahler manifold. This    will allow us
to  strengthen the theorem above by making it into a topological statement.
We will prove indeed in Section \ref{secfurther} the following result.
\begin{theo} \label{theotop} There exists a compact K\"ahler  manifold  $Y$ satisfying  the following property: There is no proper flat  holomorphic  map  $\mathcal{X}\rightarrow B$, where $B$ is an irreducible     analytic space, with one smooth  fiber $\mathcal{X}_b$  homeomorphic  to $Y$ and one projective  fiber  $\mathcal{X}_0$.
\end{theo}
A compact K\"ahler  manifold  $Y$ satisfying this property  is constructed in Section \ref{sec2X1}.
\begin{rema}{\rm Even if the statement concerns the homeomorphism class of $Y$,  one cannot replace ``irreducible'' by ``connected" in such a statement, because, in the presence of   singular fibers, it is  not true  that the smooth fibers are all homeomorphic if $B$ is only assumed connected but is not irreducible. An explicit example illustrating this phenomenon is constructed in \cite{leeoguiso}.}
\end{rema}
As one can notice, an important  difference between Theorems \ref{theomain} and \ref{theotop} is that the latter does not assume the existence of a K\"ahler fiber in the family $\mathcal{X}$. We discuss in Section  \ref{secfurther1}  some general  consequences of the existence of one singular projective fiber.

 Compact K\"ahler  manifolds $X$ or $Y$  satisfying these  conclusions  are  obtained using a  variant of the main constructions in \cite{voisin},  \cite{voisinjdg}. We review the  constructions  in Section \ref{sec2}.
 The fact that the examples $X$   constructed in Section \ref{sec2X}, (resp. the examples  $Y$ constructed in Section \ref{sec2X1},) satisfy   the conclusion of  Theorem \ref{theomain} (resp. Theorem \ref{theotop})  follows from Theorem  \ref{theomainavechypot} (resp. Theorem \ref{theonewtop}) below. To explain this, let us first  recall the strategy of \cite{voisin}, \cite{voisinjdg} using the formalism  presented in  \cite{voisinMA}. We  introduce  the following   definition  (see \cite{voisinMA})
  \begin{Defi}\label{defiHScohalg} (1) A  Hodge structure on the cohomology algebra of a compact complex manifold $X$  is the data of a (effective)  weight $k$  Hodge structure on each cohomology group $H^k(X,\mathbb{Q})$, such that
the cup-product map
$$H^k(X,\mathbb{Q})_{}\otimes  H^l(X,\mathbb{Q})_{}\rightarrow H^{k+l}(X,\mathbb{Q})_{}$$
is a morphism of Hodge structures for each pair of integers $k,\,  l$.

(2)  The Hodge structure is said polarizable if for each $k$, the weight $k$ Hodge structure on $ H^k(X,\mathbb{Q})$ is polarizable.
\end{Defi}

 If $X$ is a compact K\"ahler manifold, its cohomology algebra carries a Hodge structure as defined in (1), since the cup-product on $X$  maps
 $H^{p,q}(X)\otimes H^{p',q'}(X)$ to $H^{p+p',q+q'}(X)$. If $X$ is a projective complex manifold, it carries a polarizable Hodge structure on its cohomology algebra as in  the above definition, where the polarization is constructed using a polarization $l=c_1(L)$ on $X$, where $L$ an ample line bundle on $X$. The non-existence of a polarizable Hodge structure on the cohomology algebra is the   criterion  used in the papers \cite{voisin}, \cite{voisinjdg} to detect manifolds not homeomorphic to projective complex manifolds. We constructed in these papers   some  compact K\"{a}hler manifolds having  the property that their cohomology algebra does not admit a polarizable  Hodge structure. The manifolds $X$ and $Y$ that we will construct satisfy this property (see Theorem \ref{prorigid}). Theorem \ref{theomain} will thus follow by contradiction  from  the following result  (see also  Theorem \ref{theomainavechypot1} for a slightly stronger version)
\begin{theo}\label{theomainavechypot} Let $X$ be a compact K\"ahler manifold. Assume

(i)  The Hodge structure on the cohomology algebra    $H^*(X,\mathbb{Q})$ is  rigid under small  deformations of $X$.

(ii) The automorphisms of the Hodge structure on the cohomology algebra $H^*(X,\mathbb{Q})$   are semi-simple.

Then, if  $X$ admits a projective specialization, the cohomology algebra  $H^*(X,\mathbb{Q})$  admits a polarizable Hodge structure.
\end{theo}

Similarly, Theorem \ref{theotop}  will be obtained as a consequence of the following
\begin{theo} \label{theonewtop}  Let $Y$ be a compact complex manifold. Assume the automorphisms of the cohomology algebra of $Y$ are semi-simple. Then, if $Y$ admits a projective specialization, the cohomology algebra of $Y$ carries a  Hodge structure   which is polarizable.
\end{theo}

Note that, in Theorem \ref{theonewtop}, the assumptions on the cohomology algebra  are much stronger than in  Theorem \ref{theomainavechypot},  since in the former we are considering all the automorphisms of the cohomology algebra, while in the latter, only the automorphisms preserving also  the given Hodge structure are considered.  The semi-simplicity properties of the automorphisms of the cohomology algebra of the compact K\"ahler manifolds $X$ and $Y$ appearing respectively  in Theorem \ref{theomainavechypot} and Theorem \ref{theonewtop} will be established respectively in Sections \ref{sec2X} and  \ref{sec2X1}.

We will prove  Theorem \ref{theomain} in Section \ref{secprooftheomain} and   Theorem \ref{theotop} in Section  \ref{secfurther}.  The proofs rely on Steenbrink's work \cite{steenbrink} on limit Hodge structures, that works  using only the projectivity of the central fiber. We will also discuss in Section \ref{secfurther1} some further generalities on compact complex manifolds admitting a projective degeneration.

\vspace{0.5cm}

{\bf Thanks.} {\it  I thank Aknazar  Kazhymurat for asking me  this question and Radu Laza for helpful discussion.}

\section{The examples \label{sec2}}
\subsection{The manifold $X$ \label{sec2X}}
We will discuss here  the construction in dimension $7$ for  simplicity.   The case of dimension $>7$ will follow by taking a product with projective space and applying  similar arguments. In dimension $7$, a  manifold $X$ satisfying the assumptions of  Theorem \ref{theomainavechypot}  will be obtained by a slight variation of the examples  in  \cite{voisinjdg}.
We start from a $2$-dimensional complex torus $T$ admitting an endomorphism $\phi$ which is an isogeny  and  whose action
$\psi=\phi_*$ on $H_1(T,\mathbb{Q})$ has the property that the Galois group of the field $\mathbb{Q}[\psi]$ is as big as possible, that is, acts on the eigenvalues of $\psi$  as the symmetric group of $4$ elements. It is clear that such a  complex torus equipped with an endomorphisms  is rigid, and it is proved in \cite{voisin} that such a torus is not projective (this does not contradict \cite{catanese} which considers a finite group action).  Let $\widehat{T}:={\rm Pic}^0(T)$. Denote by $\mathcal{L}$  the Poincar\'{e} divisor on $T\times \widehat{T}$, and let  $\mathcal{L}_\phi:=(\phi,Id_{\widehat{T}})^*\mathcal{L}$. Let $t\in T,\,\,\hat{t}\in\widehat{T}$ be two points.
Let $W=\mathbb{P}^1\times T\times \widehat{T}$. We choose two distinct points $t_1,\,t_2\in\mathbb{P}^1$. $W$ contains then two disjoint complex tori, namely
$t_1\times T\times \hat{t}$ and $t_2\times{t}\times\widehat{T}$. We first blow them up in $W$, getting a compact K\"ahler manifold $\widetilde{W}$.
Next, using the same notation   $\mathcal{L},\,\mathcal{L}_\phi$  for the line bundles pulled-back to $\widetilde{W}$, we  construct the $\mathbb{P}^1$-bundles
\begin{eqnarray}
\label{eqtardiveE} P_1:=\mathbb{P}(\mathcal{E}_1),\,\,\mathcal{E}_1:=\mathcal{O}_{\widetilde{W}}\oplus\mathcal{L},\,\,\,P_2:=\mathbb{P}(\mathcal{E}_2),\,\,
\mathcal{E}_2:=\mathcal{O}_{\widetilde{W}}\oplus\mathcal{L}_\phi
\end{eqnarray}
over  $\widetilde{W}$.
Note that the $\mathbb{P}^1$-bundles $P_1,\,P_2$ over $\widetilde{W}$ have   sections $\sigma_1$, resp. $\sigma_2$,  given by the
surjective morphisms $\mathcal{O}_{\widetilde{W}}\oplus\mathcal{L}\rightarrow  \mathcal{L}$, resp. $\mathcal{O}_{\widetilde{W}}\oplus\mathcal{L}_\phi\rightarrow  \mathcal{L}_\phi$. It follows that the
fibered product $ P_1\times_{\widetilde{W}}P_2\rightarrow  \widetilde{W}$ has a section $\sigma=(\sigma_1,\sigma_2)$.
The compact K\"ahler manifold $X$ will be defined as the blow-up  of the fibered product of $P_1$ and $P_2$ along $\Sigma:=\sigma(\widetilde{W})$:
\begin{eqnarray} \label{eqdefX} X:= \widetilde{P_1\times_{\widetilde{W}}P_2}^{\Sigma}.
\end{eqnarray}
We now prove a few properties that will be needed in the next sections.
\begin{lemm}\label{proajoute1}  The Hodge structure on the cohomology algebra   $H^*(X,\mathbb{Q})$ is  rigid under small  deformations of $X$.
\end{lemm}
By ``rigid under small  deformations of $X$", we mean  that  any  flat holomorphic   proper holomorphic map  $\mathcal{X}\rightarrow B$  to a  pointed analytic space  $(B,b_0)$, with $\mathcal{X}_{b_0}\cong X$ induces near $b_0$ the trivial variation of Hodge structures. (As $X$ is K\"ahler, the fibers $\mathcal{X}_b$ are also K\"ahler for $b$ close to $b_0$.)

\begin{proof}[Proof of Lemma \ref{proajoute1}] We observe that the cohomology algebra of $X$ is generated in degree $\leq 2$, so it suffices to prove rigidity for the induced Hodge structures on $H^1(X,\mathbb{Q})$ and $H^2(X,\mathbb{Q})$. We have
\begin{eqnarray}\label{eqlistpourh2} H^2(X,\mathbb{Q})=\bigwedge^2H^1(X,\mathbb{Q})\oplus \mathbb{Q}d\oplus \mathbb{Q}\hat{d}\oplus \mathbb{Q} h\oplus \mathbb{Q}H_1\oplus \mathbb{Q}{H}_2\oplus \mathbb{Q}e,\end{eqnarray}
where $e$ is the class of the exceptional divisor $E$  over $\Sigma$, $d$, $\hat{d}$ are the pull-backs to $X$ of the classes of the  exceptional divisors of $ \widetilde{W}$, $h$ is the pull-back to $X$ of the class
$c_1(\mathcal{O}_{\mathbb{P}^1}(1))$ via the composite  morphism
$$X\rightarrow P_1\times_{\widetilde{W}}P_2\rightarrow \widetilde{W}\rightarrow W\stackrel{{\rm pr}_1}{\rightarrow} \mathbb{P}^1,$$ and
$H_1$, ${H}_2$ are the pull-backs to $X$  of  the  first Chern classes $c_1(\mathcal{O}_{P_i}(1))$ of the projective bundles $P_i\rightarrow \widetilde{W}$.

Given a  deformation
$\mathcal{X} \rightarrow B$ of $X\cong \mathcal{X}_0$, we can apply  the Kodaira stability theorems \cite{kodstab}. They say that a small deformation of a  compact complex  manifold  which is the blow-up $\widetilde{Y}\rightarrow Y$ of a manifold $Y$  along a submanifold $\Sigma\subset Y$ is the blow-up of a deformation $Y_t$  of $Y$ along a deformation $\Sigma_t\subset Y_t$ of $\Sigma$; furthermore a small deformation of a projective bundle fibration $P\rightarrow Y$ is a projective bundle fibration $P_t\rightarrow Y_t$. This implies in our case   that for $b$ close to $b_0$,
there is a divisor $E_b\subset \mathcal{X}_b$ which is a deformation of $E$ and is  the  exceptional divisor of a contraction
$\mathcal{X}_b\rightarrow Q_b$, where $Q_b$ is a deformation of $Q_{b_0}=P_1\times_{\widetilde{W}} P_2$. Furthermore, still  by \cite{kodstab}, the $\mathbb{P}^1\times \mathbb{P}^1$-fibration on $Q_{b_0}$ deforms to a $\mathbb{P}^1\times \mathbb{P}^1$-fibration on $Q_b$, hence we get
a deformation $Q_b\rightarrow \widetilde{W}_b$ of the morphism $P_1\times_{\widetilde{W}}P_2\rightarrow  \widetilde{W}$. Finally and again by  Kodaira stability theorem, the manifold $\widetilde{W}_b$ for $b$ close to $b_0$ contains contractible divisors which are deformations of the
exceptional  divisors of $\widetilde{W}$ over $W$. Thus for $b$ close to $b_0$,  $\widetilde{W}_b$ contracts to a compact complex manifold which is a deformation $W_b$ of the complex manifold  $\mathbb{P}^1\times T\times \widehat{T}$, hence is  a $\mathbb{P}^1$-fibration over a complex torus $K_b$ which is itself  a deformation of $K_{b_0}=T\times \widehat{T}$.   Furthermore, the complex manifold $W_b$ contains  the images of the two contracted divisors, which are complex submanifolds obtained as small  deformations  of  $T\times \hat{t}\subset W$ and $t\times \widehat{T}\subset W$. They are thus $2$-dimensional complex tori
$T_b$, ${T}'_b$ contained in $W_b$,  which map to  $2$-dimensional subtori $T_b,\,T'_b$ of $K_b$ intersecting in one  point. We thus conclude that
$K_b\cong T_b\oplus T'_b$.

Moreover the image
$\Sigma\subset P_1\times_{\widetilde{W}}P_2$  of the contracted  divisor $E$ deforms to the image $\Sigma_b\subset Q_b$ of the contracted divisor $E_b$, hence the $\mathbb{P}^1\times \mathbb{P}^1$- fibration  $Q_b\rightarrow W'_b$ has a section for $b$ close to $b_0$. It follows that there are  vector bundles
$\mathcal{E}_{1,b}$, $\mathcal{E}_{2,b}$ on  $ \widetilde{W}_b$ obtained as deformations of    the vector bundles $\mathcal{E}_1$, $\mathcal{E}_2$ of (\ref{eqtardiveE})  on $\widetilde{W}$,  such that $$Q_b\cong \mathbb{P}(\mathcal{E}_{1,b})\times_{\widetilde{W}_b}\mathbb{P}(\mathcal{E}_{2,b}).$$
Finally, by taking the  determinants of $\mathcal{E}_{1,b}$, $\mathcal{E}_{2,b}$,  we conclude that there are holomorphic  line bundles $\mathcal{L}_b$, $\mathcal{L}_{\phi,b} $ on $\widetilde{W}_b$ which are respective deformations of
$\mathcal{L}$, $\mathcal{L}_{\phi} $ on $\widetilde{W}$. As these line bundles are pulled-back at $b_0$ from line bundles on
$K_{b_0}$, the same is true of  the line bundles  $\mathcal{L}_b$, $\mathcal{L}_{\phi,b} $ on $\widetilde{W}_b$.  In conclusion we proved that
the complex torus $K_b\cong T_b\oplus T'_b$ carries two holomorphic line bundles
$\mathcal{L}_b$, $\mathcal{L}_{\phi,b}$ deforming $\mathcal{L}$ and $\mathcal{L}_\phi$ respectively.
Using $\mathcal{L}_b$, one concludes that $T'_b\cong \widehat{T_b}$, and using  $\mathcal{L}_{\phi,b}$, one concludes that
$T_b$ carries an endomorphism deforming $\phi$. It follows that $T_b\cong T$ since the pair $(T,\phi)$ is rigid.
We thus proved that the Hodge structure on  $H^1(\mathcal{X}_b,\mathbb{Q})$ is isomorphic to the Hodge structure on
$H^1(X ,\mathbb{Q})$. Furthermore, by the above arguments,  all the degree $2$ Hodge classes on $X$  listed in (\ref{eqlistpourh2}) deform to
degree $2$ Hodge classes on $\mathcal{X}_b$, so  (\ref{eqlistpourh2}) shows that  the Hodge structure on  $H^2(\mathcal{X}_b,\mathbb{Q})$ is isomorphic to the Hodge structure on
$H^2(X ,\mathbb{Q})$.
\end{proof}
\begin{lemm}\label{leautoalgcoh} The automorphisms of the Hodge structure on the  cohomology algebra $H^*(X,\mathbb{Q})$ are semisimple.
\end{lemm}
\begin{proof} The automorphisms of the Hodge structure on $H^1(X,\mathbb{Q})=H^1(T,\mathbb{Q})\oplus H^1(T,\mathbb{Q})^*$ are of the form   $$ \begin{pmatrix}
 \psi_1& 0
 \\
0&\psi_2
\end{pmatrix},
$$
where $\psi_1$ resp.  $\psi_2$  are   automorphisms of  the Hodge structure on   $H^1(T,\mathbb{Q})$,  resp. $ H^1(T,\mathbb{Q})^*$. Indeed, the two tori $T,\,\widehat{T}$ are simple and not isogenous. Next, using the fact that $^t\psi$, resp. $\psi$ acts
in an irreducible way as an automorphism of the Hodge structure on  $H^1(T,\mathbb{Q})$, resp. $ H^1(T,\mathbb{Q})^*$, we get that
$\psi_1\in\mathbb{Q}[^t\psi]$, resp.  $\psi_2\in\mathbb{Q}[\psi]$. As the endomorphism $\psi$ is semisimple, it follows that any automorphism of the Hodge structure on  $H^1(X,\mathbb{Q})$ is semisimple.

It remains to prove semisimplicity of the action on the higher degree cohomology groups. As the cohomology algebra of $X$ is generated in degree $\leq 2$, it suffices to prove the semisimplicity on degree $2$ cohomology. As the automorphisms of the Hodge structure  on $H^2(X,\mathbb{Q})$ preserve the space of Hodge classes,  using the decomposition (\ref{eqlistpourh2}), we only need to prove
\begin{claim} An automorphism $\gamma$ of the  Hodge structure on the cohomology algebra of $X$  has a power $\gamma^r$, $r>0$ which preserves the space
$\langle d,\,\hat{d},\, h,\,H_1,\,{H}_2,\,e\rangle\subset H^2(X,\mathbb{Q})$  and  acts in a semisimple way on it.
\end{claim}
\begin{proof}  The automorphism $\gamma$ acts on $H^*(X,\mathbb{Q})$ preserving the algebra structure and the Hodge structure.
In particular, it preserves the subalgebra generated by $H^1(X,\mathbb{Q})$ and the subspace ${\rm Hdg}^2(X)$, hence it preserves ${\rm Hdg}^2(X)\cap \bigwedge^2H^1(X,\mathbb{Q})$. We claim that a power of $\gamma$ preserves up to a scalar multiple the classes $h,\,H_1,\,H_2$. For the class  $h$, the reason is that
$h^2=0$ in $H^4(X,\mathbb{Q})$,   and $h$ is an isolated solution of this set of  equations in $\mathbb{P}({\rm Hdg}^2(X))$; moreover $\gamma$ acts on $\mathbb{P}({\rm Hdg}^2(X))$ preserving the algebraic subset  defined by these equations, since it preserves the algebra structure of $H^*(X,\mathbb{Q})$. For the classes $H_1,\,H_2$, one observes that they satisfy an equation of the form
$H_i^2=H_il_i$, for some classes $l_i\in \bigwedge^2H^1(X,\mathbb{Q})\cap {\rm Hdg}^2(X)$. These equations come from the fact that the vector bundles $\mathcal{E}_i$ of (\ref{eqtardiveE}) have $c_2=0$. Again the classes $H_i$  are isolated solutions of these equations in  $\mathbb{P}({\rm Hdg}^2(X))$ so that a power of  $\gamma$ will fix  them.
 Similarly, a power of $\gamma$ preserves up to a scalar the classes
$ d$ and $\hat{d}$, because they have the property that
the cup-product map
$d\cup,\,\hat{d}\cup : H^1(X,\mathbb{Q})\rightarrow H^3(X,\mathbb{Q})$ is not injective, and they are the only elements  of  $\mathbb{P}({\rm Hdg}^2(X))$ which do not belong to  $\mathbb{P}( \bigwedge^2H^1(X,\mathbb{Q})\cap {\rm Hdg}^2(X))$  and satisfy this condition. We conclude that a power $\gamma^r$, $r>0$, of $\gamma$ preserves the subalgebra generated by $H^1(X,\mathbb{Q})$ and
$h,\,H_1,\,H_2,\,d,\,\hat{d}$, that is, the cohomology subalgebra
$$H^*(P_1\times_{\widetilde{W}}P_2,\mathbb{Q})\stackrel{\tau}{\hookrightarrow}  H^*(X,\mathbb{Q}).$$
But then, using the fact that the morphism $\tau: X\rightarrow P_1\times_{\widetilde{W}}P_2$ is birational,   $\gamma^r$  also preserves the class $e$, since it generates the kernel of the morphism $\tau_*:H^2(X,\mathbb{Q})\rightarrow  H^2(P_1\times_{\widetilde{W}}P_2,\mathbb{Q})$. One uses here the fact that $\tau_*$ is determined by $\tau^*$ and the algebra structure, by Poincar\'{e} duality. Thus a power  $\gamma^r$ preserves $\langle h,\, H_1,\,H_2,\,d,\,\hat{d},\,e\rangle $ and acts in a diagonalized way on it. The claim is proved.
\end{proof}
We thus proved that a power of $\gamma$ acts in a semisimple way on $H^*(X,\mathbb{Q})$, which implies that $\gamma$ acts in a semisimple way since $\gamma$ is an automorphism.
\end{proof}
\subsection{The manifold $Y$ \label{sec2X1}}
The compact K\"ahler manifold $X$ constructed in the previous section has the property that the automorphisms of the cohomology algebra $H^*(X,\mathbb{Q})$ preserving its  Hodge structure are semisimple (see Lemma \ref{leautoalgcoh}). This is the assumption appearing in  Theorem \ref{theomainavechypot}. We are now going to construct a compact K\"ahler manifold $Y$  having the property that all  the automorphisms of  its cohomology algebra $H^*(Y,\mathbb{Q})$ are semisimple, which is the assumption appearing in  Theorem \ref{theonewtop}.

We will discuss here  the construction in dimension $5$ for  simplicity. The case of dimension $>5$ will follow by taking a product with projective space and applying  similar arguments.
In dimension $5$, a manifold $Y$ satisfying  the assumptions of  Theorem \ref{theonewtop} will be obtained by a small variant of \cite{voisin}. We start as in the previous section  from a $2$-dimensional complex torus $T$ admitting an endomorphism $\phi$ which is an isogeny and  whose action
$\psi=\phi_*$ on $H_1(T,\mathbb{Q})$ has the property that the Galois group of the field $\mathbb{Q}[\psi]$  acts on the eigenvalues of $\psi$  as the symmetric group of $4$ elements.

We start with $W:=T\times T\times \mathbb{P}^1$ and choose four distinct elements $t_1,\ldots,t_4$ in $\mathbb{P}^1$. Let $t\in T$ be chosen.
 $W$ contains  the following four disjoint complex submanifolds, all isomorphic to $T$:

\begin{eqnarray}\label{eqcenters} T\times t\times t_1,\,\,t\times T\times t_2,\,\,\Delta_T\times t_3,\,\, \Gamma_\phi\times t_4,
\end{eqnarray}
where as usual $\Delta_T$ is the diagonal of $T$ and $\Gamma_\phi$ is the graph of $\phi$.

We will denote by $Y$ the compact K\"ahler manifold  obtained by blowing-up $W$ along these four submanifolds. The tiny difference with the construction of \cite{voisin} is that the product with $\mathbb{P}^1$ and the choice of four distinct points $t_i$ in $\mathbb{P}^1$ makes the four centers (\ref{eqcenters}) of blow-up disjoint, while their images   in $T\times T$  are not disjoint. This forced us in \cite{voisin}  to blow-up first their intersection points in order  to separate them before blow-up, making the cohomological computations heavier.

We will need   the following  statement concerning the Hodge structures on the cohomology
algebra of $Y$.
\begin{prop} \label{proajoute}   The automorphisms of  the cohomology algebra   $H^*(Y,\mathbb{Q})$    are semi-simple.
\end{prop}
\begin{proof}    As the cohomology algebra
 $H^*(Y,\mathbb{Q})$ is generated in degree $\leq 2$, it suffices to prove that the action of an automorphism $\gamma$ of $H^*(Y,\mathbb{Q})$  is semi-simple on $H^1(Y,\mathbb{Q})$ and $H^2(Y,\mathbb{Q})$. We have  $H^2(Y,\mathbb{Q})={\rm Hdg}^2(Y)\oplus \bigwedge^2H^1(Y,\mathbb{Q})$, where the space ${\rm Hdg}^2(Y)$ of Hodge classes on $Y$ is generated by the class $h$ which is the pull-back to $Y$ of the class ${\rm pr}_1^*c_1(\mathcal{O}_{\mathbb{P}^1}(1))$ on $W$ and the classes $d_i$ of the four exceptional divisors of $Y$ over $W$. Clearly the subspace
$\bigwedge^2H^1(Y,\mathbb{Q})\subset H^2(Y,\mathbb{Q})$ is preserved by any automorphism of  the cohomology  algebra.
We first prove
\begin{lemm}\label{lepourendosurhdg}  Any endomorphism $\gamma$  of   the cohomology algebra $H^*(Y,\mathbb{Q})$ preserves the subspace
 $${\rm Hdg}^2(Y)\subset H^2(Y,\mathbb{Q})$$ and  acts in a semi-simple  way on it.
\end{lemm}
\begin{proof}
 The subspace
$\mathbb{Q}h$ is  preserved by $\gamma$ because, up to a coefficient, $h$ is the only  class in $H^2(Y,\mathbb{Q})$ which is not in $\bigwedge^2H^1(Y,\mathbb{Q})$ and satisfies $h^2=0$.  Next we claim that   the $\mathbb{Q}$-vector  subspace $\langle d_i\rangle\subset {\rm Hdg}^2(Y)$ is stable under $\gamma$.  This follows indeed from the structure of the cohomology algebra: Looking at the cup-product map
$$\alpha\cup: H^1(Y,\mathbb{Q})\rightarrow H^3(Y,\mathbb{Q})$$ for $\alpha\in   H^2(Y,\mathbb{Q})$, we see that the $d_i$'s are up to a scalar multiple  the only classes not in  $\bigwedge^2H^1(Y,\mathbb{Q})$ for which the cup product map
$d_i\cup: H^1(Y,\mathbb{Q})\rightarrow H^3(Y,\mathbb{Q})$ has a kernel of dimension $\geq 4$. Hence $\gamma$ preserves the subspace $\langle d_i\rangle\subset H^2(Y,\mathbb{Q})$ and thus  preserves ${\rm Hdg}^2(Y)$.
Finally,  exactly  for  the same reason,  $\gamma$ acts projectively  on the space $\mathbb{P}(\langle d_i\rangle)$  by permuting the $d_i$'s, so that a power of $\gamma$  acts  on the projective space $\mathbb{P}(\langle d_i\rangle)$  fixing the $d_i$'s.
We thus proved  that  the  action of a power of  $\gamma$ on  the vector   space ${\rm Hdg}^2(Y)$ is diagonalizable, which implies  semisimplicity of the action of $\gamma$ on ${\rm Hdg}^2(Y)$, since $\gamma$ is an automorphism.
\end{proof}

For the action of $\gamma$ on degree $1$ cohomology, we  next  prove
\begin{lemm}\label{profinaloupresque}   Any automorphism
$\gamma$ of the cohomology algebra  $ H^*(Y,\mathbb{Q})$ has a power $\gamma^r,\,r>0$, acting  on $H^1(Y,\mathbb{Q})=H^1(T,\mathbb{Q})^2$ via a matrix  \begin{eqnarray}\label{eqmatrix} \begin{pmatrix}
 \psi_1& 0
 \\
0&\psi_1
\end{pmatrix},
\end{eqnarray}  where $\psi_1\in \mathbb{Q}[^t\psi]$.
\end{lemm}

\begin{proof} The proof of Lemma \ref{lepourendosurhdg}   shows that $\gamma$ preserves the class $h$ and the $\mathbb{Q}$-vector space $K:=\langle d_i\rangle$. We also  saw that a power  $\gamma^r$ for some $r>0$ acts  on  $\mathbb{P}(K)$ fixing the $d_i$. It then follows that
the action of $\gamma^r $ on $H^1(Y,\mathbb{Q})$ preserves the four subspaces
$L_i:={\rm Ker}\,(d_i\cup:  H^1(Y,\mathbb{Q})\rightarrow H^3(Y,\mathbb{Q}))$. We have
$$H^1(Y,\mathbb{Q})=L_1\oplus L_2,$$
and this decomposition is preserved by $\gamma^r$. Furthermore
 the subspace $L_3\subset L_1\oplus L_2$ defines an isomorphism $L_1\cong L_2$, so we can see
 $L_4\subset  L_1\oplus L_2\cong L_1\oplus L_1$ as the graph of an automorphism $\psi'$ of $L_1$, that is easily seen to be conjugate to $^t\psi$. As $\gamma^r$ preserves the $L_i$, it follows from the above analysis  that $\gamma^r $ acts on $L_1$ and $L_2$ via the same  automorphism $\psi_1$, that has the property that $\begin{pmatrix}
 \psi_1& 0
 \\
0&\psi_1
\end{pmatrix}$  maps the graph of $^t\psi$ to itself. Equivalently, $\psi_1$  commutes with $^t\psi$. As $\psi$ has only simple  eigenvalues, it follows that $\psi_1\in \mathbb{Q}[^t\psi]$.
\end{proof}
As $\psi_1$ is semisimple, Lemma \ref{profinaloupresque} implies that  the action of a power $\gamma^r$, $r>0$ of  $\gamma$ on $H^1(Y,\mathbb{Q})$ is semisimple, and it follows  that  the action of $\gamma^r$ on $H^1(Y,\mathbb{Q})$ and  $\bigwedge^2H^1(Y,\mathbb{Q})$ is semisimple. Using Lemma \ref{lepourendosurhdg}, we conclude that the action of $\gamma^r$ on $H^2(Y,\mathbb{Q})$ is semisimple.
This concludes the proof of Proposition \ref{proajoute}.
\end{proof}
To conclude this section, we note for future use
that  the compact  K\"ahler   manifolds $X$ and  $Y$ constructed above satisfy the following property  (see \cite{voisin}, \cite{voisinjdg}).
\begin{theo} \label{prorigid}
 For any Hodge structure on   the cohomology algebra $H^*(X,\mathbb{Q})$, resp.  $H^*(Y,\mathbb{Q})$, the induced weight $1$ Hodge structure on $H^1(X,\mathbb{Q})$, resp. $H^1(Y,\mathbb{Q})$  is not polarizable.
\end{theo}
The existence of a polarizable Hodge structure on the cohomology algebra  is  the necessary  criterion  used  in \cite{voisin} for a compact K\"ahler manifold  to be homeomorphic to a complex projective manifold. Theorem \ref{prorigid} is proved in \cite{voisin} for the manifold $Y$. The case of $X$ is treated by similar arguments. In both cases, the idea is that the structure of the cohomology algebra forces a Hodge structure on the cohomology algebra to have its degree $1$ part isomorphic to a direct sum $H_1\oplus H_2$ of weight $1$ Hodge structures admitting an endomorphism conjugate to $\psi$ or $^t\psi$, which makes them not polarizable.

\section{Proof of Theorem  \ref{theomain}}
\subsection{Preliminary results}
We start with a general statement concerning proper holomorphic fibrations with one K\"ahler fiber.
\begin{prop} \label{theoeasycons}  Let  $\pi: \mathcal{X}\rightarrow B$ be a proper flat  holomorphic  map, where $B$ is an irreducible  analytic space. Assume one smooth  fiber $\mathcal{X}_t$ is K\"ahler.  Then there exists  a dense open subset $U\subset B$ which is the complement of a closed real analytic subset of $B$, such that the fiber $\mathcal{X}_b$, for $b\in U$, carries a Hodge structure on its  cohomology algebra.
\end{prop}
\begin{proof} For any complex manifold $X$ and for any integer $k$, there is a ``Hodge  filtration" on the cohomology $H^k(X,\mathbb{C})$ defined
as
\begin{eqnarray}  F^rH^k(X,\mathbb{C})={\rm Im}\,(\mathbb{H}^k(X,\Omega_X^{\bullet\geq r})\rightarrow \mathbb{H}^k(X,\Omega_X^{\bullet})).
\label{eqhodgefilt}
\end{eqnarray}
These filtrations for various $k$  are  obviously compatible with the cup-product map in the sense that
$F^rH^k\cup F^s H^l\subset F^{r+s}H^{k+l}$.
In order to prove that there is a Hodge structure on the cohomology algebra of $X$, it thus  suffices to show that these Hodge filtrations define Hodge structures, which is equivalent to saying that they satisfy the opposite condition for all $k,\,r$
\begin{eqnarray}\label{eqopp}
F^rH^k(X,\mathbb{C})\oplus \overline{F^{k-r+1} H^k(X,\mathbb{C})}\cong H^k(X,\mathbb{C}).
\end{eqnarray}
In the situation of the proposition, we know that one fiber is K\"ahler. It follows that its Fr\"olicher spectral sequence degenerates at $E_1$. This implies that the Fr\"{o}licher spectral sequence of $\mathcal{X}_b$ degenerates at $E_1$ for $b$ in a non-empty  open subset $V\subset B$ (contained in the open subset $B^0$ of $B$ parameterizing smooth fibers), which is the complement of a proper closed analytic subset (we will use the terminology ``analytic Zariski open'' in the sequel). As $B$ is irreducible, this open set  is dense.
Next, the degeneracy at $E_1$ of the   Fr\"{o}licher spectral sequence of the fibers $\mathcal{X}_b$ over $V$ implies that, denoting
$\pi_V:\mathcal{X}_V\rightarrow V$ the restriction of $\pi$ to $\mathcal{X}_V:=\pi^{-1}(V)$,
the subsheaves
$$F^r\mathcal{H}^k\subset \mathcal{H}^k,\,\,\mathcal{H}^k=\mathbb{R}^k\pi_{V*}\Omega^\bullet_{\mathcal{X}^0/B^0}=\mathbb{R}^k\pi_{V*}\mathbb{C}\otimes\mathcal{O}_{B^0}$$
are locally free subbundles satisfying the base change property, because the sheaf $\mathcal{H}^k$ is locally free.
It then follows from the constancy of the rank   of $F^rH^k(\mathcal{X}_b,\mathbb{C})$,  $ b\in V$, that the locus $U$  of points $b\in  V$
where the opposite conditions (\ref{eqopp}) are satisfied is the complement of a closed real analytic  subset of $V$.  This locus is not empty by assumption and, as  $B$ is irreducible,   $U$ is dense in $B$.
\end{proof}

\begin{coro}\label{coromono}  In the situation above, assume the compact K\"ahler fiber $\mathcal{X}_t$ has the property that the Hodge structure on its cohomology algebra is rigid under small deformations of $\mathcal{X}_t$. Then, for any integer $k$,  the monodromy action
\begin{eqnarray}\label{eqmonoact} \pi_1(B^0,t)\rightarrow {\rm Aut}\,H^k(\mathcal{X}_t,\mathbb{Q})\end{eqnarray}
acts via automorphisms of Hodge structures.
\end{coro}
\begin{proof}  The Hodge structure of the cohomology algebra of $\mathcal{X}_t$ being rigid under small deformations of $\mathcal{X}_t$, for any integer $k$,  the Hodge structure  on the degree $k$ cohomology of  the fibers $\mathcal{X}_b$ for $b\in U$ close to $b$ is locally constant. It follows by analytic continuation that the Hodge filtration, which  varies holomorphically and  is defined on the analytic Zariski dense open set $V\subset B^0$, is locally constant on $V$. As the natural map  $\pi_1(V,t)\rightarrow \pi_1(B^0,t)$ is surjective, the monodromy action  (\ref{eqmonoact})  on $H^k(\mathcal{X}_t,\mathbb{Q})$ preserves the Hodge filtration, hence the Hodge structure.
\end{proof}
\begin{rema}\label{remaHSpartout}  {\rm The proof also shows that for any $b\in V$, the  Hodge filtration on $H^k(\mathcal{X}_b,\mathbb{C})$ induces a Hodge structure on $H^k(\mathcal{X}_b,\mathbb{Q})_{}$ which is isomorphic to that of $H^k(\mathcal{X}_t,\mathbb{C})$.}
\end{rema}
\subsection{Projective specializations of compact K\"ahler manifolds \label{secprooftheomain}}
Our goal  in this section is to prove Theorem \ref{theomain}, as a consequence of  the following result on projective specializations of compact K\"ahler manifolds. This is a slight improvement of Theorem \ref{theomainavechypot}.
\begin{theo}\label{theomainavechypot1} Let $X$ be a compact K\"ahler manifold. Assume

(i)  The Hodge structure on the cohomology algebra  $H^*(X,\mathbb{Q})$ is rigid under small deformations of $X$.

(ii) The automorphisms of the Hodge structure on the cohomology algebra $H^*(X,\mathbb{Q})$ preserving a nonzero element  of  $H^2(X,\mathbb{Q})$ with nonzero top self-intersection   are semi-simple.

Then, if  $X$ admits a projective specialization, the cohomology algebra  $H^*(X,\mathbb{Q})$  admits a polarizable Hodge structure.
\end{theo}
Before giving the proof, we prove Theorem \ref{theomain}.
\begin{proof}[Proof of Theorem \ref{theomain}] We observe that  assumptions (i) and (ii) above  are satisfied by the compact K\"ahler manifold $X$ constructed in Section \ref{sec2X}. This is indeed the contents of   Proposition \ref{proajoute1} and Lemma \ref{leautoalgcoh}.  By Theorem \ref{prorigid}, the cohomology algebra of $X$ does not admit  a polarizable Hodge structure. Theorem \ref{theomainavechypot1}  thus shows that $X$ does not admit a projective specialization.
\end{proof}
\begin{proof}[Proof of Theorem \ref{theomainavechypot1}]
Let $X$ be a compact K\"ahler manifold satisfying the assumptions (i) and (ii) above.   Let $\pi:\mathcal{X}\rightarrow B$ be a flat proper holomorphic map between analytic spaces,  where $B$ is irreducible, with one fiber $\mathcal{X}_b$  isomorphic to $X$ and one fiber $\mathcal{X}_0$ projective. We can assume by desingularization that $B$ is smooth with discriminant  divisor $\mathcal{D}=B\setminus B^0$, and we can choose a disc $D\subset B$ centered at $0$  and not contained in $\mathcal{D}$. The analytic Zariski open set $V$ being defined as in the  previous section, we can even assume that $D$ is not contained in $B\setminus  V$. From now on, we will only work with the restricted family
$$\pi':\mathcal{X}':=D\times_B\mathcal{X}\rightarrow D.$$
Note that, after a finite base-change and blow-up over the central fiber, one can assume that
$\mathcal{X}'$ is smooth and  the central fiber $\mathcal{X}'_0$ is a reduced divisor with normal crossings.

For any choice of $s\in D^*\cap V$, one has the monodromy representation
$$\rho:\mathbb{Z}=\pi_1(D^*,s)\rightarrow {\rm Aut}\,H^k(\mathcal{X}_s,\mathbb{Q}),$$
hence a monodromy operator $T=\rho(\gamma)$ where $\gamma$ is a counterclockwise  loop around $0$ based at $s$.  By assumption  (i),  the assumptions of Corollary \ref{coromono} are satisfied by $X$.  Corollary \ref{coromono} and Remark \ref{remaHSpartout} imply that there is a Hodge structure on $H^*(\mathcal{X}_s,\mathbb{Q})$, isomorphic to the  Hodge structure on $H^*(X,\mathbb{Q})$, and $T$ acts on  the cohomology algebra  $H^*(\mathcal{X}_s,\mathbb{Q})$ as an automorphism of Hodge structures, which obviously also preserves the algebra structure.
We use now the fact that the central fiber is projective. There is thus a class $H\in H^2(\mathcal{X}'_0,\mathbb{Q})$ defined as the restriction  of  the hyperplane class in a projective embedding of $\mathcal{X}'_0$.
Choosing  a topological retraction $r$ of $\mathcal{X}'$ onto $\mathcal{X}'_0$, the class $r^*H_{\mid \mathcal{X}'_s}\in H^2(\mathcal{X}'_s,\mathbb{Q})$ is invariant
under $T$. We have $\int_{\mathcal{X}'_0}H^n\not=0$, $n={\rm dim}\,X$,  hence $ \int_{\mathcal{X}'_s}r^*H^n\not=0$. By assumption (ii),   the action of $T$ is thus semi-simple. We now apply the  following result (which in this context  is due to Steenbrink).
\begin{theo} (Steenbrink \cite[Theorem 2.21]{steenbrink}) \label{theounip}  Let $\mathcal{X}$ be    a  complex manifold and  $\mathcal{X}\rightarrow D$  a flat   proper holomorphic map  with smooth fibers $\mathcal{X}_s$ for $s\not=0$ and  reduced normal crossing central fiber. Then  the action of $T$ on $H^k(\mathcal{X}_s,\mathbb{Q})$ is unipotent.
\end{theo}
It follows that the eigenvalues of $T$ on $H^*(\mathcal{X}_s,\mathbb{Q})$ are equal to $1$, hence as $T$ is semi-simple, it must  be  the identity.

Let $  \widetilde{D^*}$ be the universal cover of $D^*$ and let $ \widetilde{\mathcal{X}'}^*:=\mathcal{X}'\times_D  \widetilde{D^*}$. Let
$H^*_{\rm lim}:=H^*(\widetilde{\mathcal{X}'}^*,\mathbb{Q})$. The cohomology algebra $H^*_{\rm lim}$ is isomorphic to $H^*(X,\mathbb{Q})$.
We  now use  the fact that the central fiber $\mathcal{X}'_0$ is projective. We can then apply  the following  results of Steenbrink  \cite{steenbrink}.
\begin{theo}   (i) There is a mixed Hodge structure on each $H^k_{\rm lim}$.

(ii) More precisely, $H^k_{\rm lim}$ is computed as the abutment of a (weight)  spectral sequence $_WE_1^{p,q}\Rightarrow H^{p+q}_{\rm lim}$, an each
$_WE_1^{p,q}$ carries a Hodge structure. The differentials are morphisms of Hogge structures and the  spectral sequence degenerates at $E_2$.

(iii) Let $N={\rm log}\,T$ (as $Id-T$ is nilpotent, this is a polynomial in $T$ with rational coefficients). Then for each $r\geq 0$, $N^r$ induces an isomorphism
$$N^r: {_WE_1^{-r,q+r}}\cong  {_WE_1^{r,q-r}}.$$
\end{theo}

In our case, we proved that $T=Id$, so $N=0$. It follows from (iii) that $_WE_1^{r,q}=0$ for $r>0$ or $r<0$. Thus the weight filtration is trivial and the mixed Hodge structure on each $H^k_{\rm lim}$ is pure. Finally, the Hodge structures on  $_WE_1^{0,q}$ are direct sums of Hodge structures on cohomology groups of intersections of components of the normal crossing divisor $\mathcal{X}'_0$, assuming it has global normal crossings. As these successive intersections are projective manifolds,  these Hodge structures are polarized.
We thus proved that each  $H^k_{\rm lim}$ is equipped with a polarizable Hodge structure. Looking at the construction of the Hodge filtration in \cite{steenbrink}, it is clear that these Hodge filtrations are compatible with the cup-product on cohomology, hence we get a polarizable Hodge structure on the cohomology algebra $H^*_{\rm lim}\cong H^*(X,\mathbb{Q})$.
\end{proof}

\section{Topological variant  and further remarks}
\subsection{Proof of Theorems \ref{theotop} and \ref{theonewtop}  \label{secfurther}}
\begin{proof}[Proof of Theorem   \ref{theonewtop}] The proof  uses arguments which already appeared in the previous section. First of all,
we observe that the proof of Theorem \ref{theomain} given in the previous section shows the following:
\begin{prop} \label{progenD}  Let $\pi:\mathcal{X}\rightarrow D$ be a flat proper holomorphic map, where $\mathcal{X}$ is smooth, such that the central fiber is projective and the other fibers are smooth. Let $s\in D^*$ and assume that the monodromy action $T$ on  $H^*(\mathcal{X}_s,\mathbb{Q})$ is trivial. Then the cohomology algebra $H^*(\mathcal{X}_s,\mathbb{Q})$ admits a  polarizable Hodge structure.
\end{prop}
In the situation of Theorem \ref{theonewtop}, we have  a proper flat  family $\mathcal{X}\rightarrow B$ with smooth fibers homeomorphic to $X$ and  a point $0\in B$ such that the fiber $\mathcal{X}_0$ is projective.  After restriction to a carefully chosen disc in $B$ passing through $0$, finite base change  and semistable reduction, we get
a monodromy operator $T$ acting on $H^*(\mathcal{X}'_t,\mathbb{Q})$ for $t\in D^*$. The action of $T$ obviously preserves the structure of $H^*(X,\mathbb{Q})$ as an algebra, hence it is semi-simple by our main assumption. Steenbrink's theorem \ref{theounip} then says that $T=Id$, so that Proposition \ref{progenD} applies.
\end{proof}
\begin{proof}[Proof of Theorem \ref{theotop}]  According to  Proposition \ref{proajoute}, the compact K\"ahler manifold   $Y$ constructed in Section \ref{sec2X1} satisfies the assumptions of Theorem \ref{theonewtop}. Moreover, by Theorem \ref{prorigid}, there is no polarizable  Hodge structure on the cohomology algebra of $Y$, hence $Y$ does not admit a projective specialization by Theorem \ref{theonewtop}.
\end{proof}

\subsection{Further remarks  on projective specializations \label{secfurther1}}
Consider a proper flat holomorphic map
$$\pi:\mathcal{X}\rightarrow D$$
where $D$ is the disc and $\mathcal{X}$ is smooth. We can assume that only the central fiber is singular, and  by Hironaka resolution and semistable reduction, we can  assume after base change that the central fiber is a reduced normal crossing divisor.
Let $D^*:=D\setminus\{0\}$.
We now prove
\begin{prop} \label{theodegfro} Assume the central fiber $W:=\mathcal{X}_0$ is projective. Then the Fr\"{o}licher spectral sequence of
$\mathcal{X}_t$ degenerates at $E_1$  for $t$ in an analytic  Zariski open set of $D^*$.
\end{prop}
\begin{proof}  The Fr\"{o}licher spectral sequence has a relative version over $D$, using the relative   logarithmic holomorphic de Rham complex $\Omega^\bullet_{\mathcal{X}/D}({\rm log}\,W)$. According to Steenbrink (see \cite[Theorem 2.18]{steenbrink}), the coherent sheaf
$\mathbb{R}^k\pi_*\Omega^\bullet_{\mathcal{X}/D}({\rm log}\,W)$ on $D$ is locally free and satisfies base change for any $k$.
There is a relative Fr\"{o}licher spectral sequence
\begin{eqnarray}\label{eqrelss} E_1^{p,q}=R^q\pi_*\Omega^p_{\mathcal{X}/D}({\rm log}\,W)\Rightarrow  \mathbb{R}^k\pi_*\Omega^\bullet_{\mathcal{X}/D}({\rm log}\,W).
\end{eqnarray}
The relative logarithmic holomorphic de Rham complex restricts on each fiber $\mathcal{X}_t$, $t\not=0$, to the holomorphic de Rham complex
of $\mathcal{X}_t$.  Moreover, when $W$ is projective,  the Fr\"{o}licher spectral sequence of the restricted complex $(\Omega^\bullet_{\mathcal{X}/D}({\rm log}\,W))_{\mid W}$ degenerates at $E_1$ by \cite[Corollary 4.20]{steenbrink}. Upper semicontinuity and local freeness of $\mathbb{R}^k\pi_*\Omega^\bullet_{\mathcal{X}/D}({\rm log}\,W)$ then imply by standard arguments that the Fr\"{o}licher spectral sequence (\ref{eqrelss}) degenerates at $E_1$ for $t$ in an analytic Zariski open neighborhood of $0$ in $D$.
\end{proof}
\begin{rema}\label{remautile}  {\rm  We also deduce, using  the same arguments as above,  that the sheaves $R^q\pi_*\Omega_{\mathcal{X}/D}^p$ are locally free and satisfy base-change  in an analytic-Zariski open neighborhood of $0$ in $D$, and that the subsheaves
$$F^r\mathbb{R}^k\pi_*\Omega^{\bullet\geq r}_{\mathcal{X}/D}({\rm log}\,W):={\rm Im}\,(\mathbb{R}^k\pi_*\Omega^{\bullet\geq r}_{\mathcal{X}/D}({\rm log}\,W)\rightarrow \mathbb{R}^k\pi_*\Omega^\bullet_{\mathcal{X}/D}({\rm log}\,W))$$
are locally free subsheaves and satisfy base change.}
\end{rema}
\begin{rema}{\rm  Proposition \ref{theodegfro} has some similarities with Proposition 1 of \cite{kerrlazasaito}. I thank the referee for mentioning this reference. }
\end{rema}
We do not know the answer to the following questions.  Assume as before  the central fiber $Y:=\mathcal{X}_0$ of a flat proper holomorphic map $\mathcal{X}\rightarrow D$, with $\mathcal{X}$ smooth,  is projective.

\begin{question} Does the Hodge filtration defined above on the cohomology of the fibers $\mathcal{X}_t$ define  Hodge structures for $t$ close to $0$?
\end{question}
Assuming the answer to this question is yes, the following question remains
\begin{question}  Are the Hodge structures  on  the cohomology of the fibers  $\mathcal{X}_t$ for $t$ close to $0$  polarizable with real coefficients?
\end{question}
By this we mean precisely the following: Does there exist a class $\alpha_t\in H^{1,1}(\mathcal{X}_t)_\mathbb{R}$ satisfying the hard Lefschetz property, so that the associated Lefschetz intersection pairing satisfies the Hodge-Riemann bilinear relations.
\begin{question}  Does  the fiber  $\mathcal{X}_t$  belong to the Fujiki class $\mathcal{C}$? Are there examples where some or all  fibers
$\mathcal{X}_t$ are not K\"ahler for $t$ arbitrarily close to $0$?
\end{question}

CNRS, Institut de Math\'{e}matiques de Jussieu-Paris rive gauche

claire.voisin@imj-prg.fr
    \end{document}